\documentclass[a4paper,12pt]{article}
\usepackage{amsmath,amssymb,amsthm,latexsym}

\theoremstyle{plain}
\newtheorem{theorem}{Theorem}
\newtheorem{cor}[theorem]{Corollary}

\theoremstyle{definition}
\newtheorem*{defi}{Definition}
\newtheorem{rem}[theorem]{Remark}

\title{The double competition multigraph of a digraph}

\author{
{\sc Jeongmi PARK}
\footnote{
Department of Mathematics, 
Pusan National University, 
Busan 609-735, Korea. 
{\it E-mail}: {\tt jm1015@pusan.ac.kr}}
\and 
{\sc Yoshio SANO}
\thanks{This work was supported by JSPS KAKENHI Grant Numbers 25887007, 15K20885.}
\footnote{
Division of Information Engineering, 
Faculty of Engineering, Information and Systems,  
University of Tsukuba, 
Ibaraki 305-8573, Japan.
{\it E-mail}: {\tt sano@cs.tsukuba.ac.jp}} 
}

\date{}

\begin{document}

\maketitle

\begin{abstract}
In this article, 
we introduce the notion of the double competition 
multigraph of a digraph. 
We give characterizations of the double competition 
multigraphs of arbitrary digraphs, 
loopless digraphs, reflexive digraphs, 
and acyclic digraphs 
in terms of edge clique partitions of the multigraphs. 
\end{abstract}




\section{Introduction}

The competition graph of a digraph is defined to be 
the intersection graph of the family of the out-neighborhoods of 
the vertices of the digraph (see \cite{MM99} for intersection graphs). 
A \emph{digraph} $D$ is a pair $(V(D),A(D))$ of a set $V(D)$ 
of \emph{vertices} and a set $A(D)$ of ordered pairs of vertices, 
called \emph{arcs}. 
An arc of the form $(v,v)$ is called a \emph{loop}. 
For a vertex $x$ in a digraph $D$,  
we denote the \emph{out-neighborhood} of $x$ in $D$ by 
$N^+_D(x)$ and the \emph{in-neighborhood} of $x$ in $D$ by 
$N^-_D(x)$, i.e., 
$N^+_D(x) := \{v \in V(D) \mid (x,v) \in A(D) \}$ 
and 
$N^-_D(x) := \{v \in V(D) \mid (v,x) \in A(D) \}$. 
A \emph{graph} $G$ is a pair $(V(G),E(G))$ of a set $V(G)$ 
of \emph{vertices} and a set $E(G)$ of unordered pairs of vertices, 
called \emph{edges}. 
The \emph{competition graph} of a digraph $D$ 
is the graph which has the same vertex set as $D$ 
and has an edge between two distinct vertices $x$ and $y$ 
if and only if $N^+_D(x) \cap N^+_D(y) \neq \emptyset$. 
R. D. Dutton and R. C. Brigham \cite{DB83} and 
F. S. Roberts and J. E. Steif \cite{RS83} 
gave characterizations of competition graphs 
by using edge clique covers of graphs. 
The notion of competition graphs was 
introduced by J. E. Cohen \cite{Cohen} in 1968 
in connection with a problem in ecology, and 
several variants and generalizations of competition graphs 
have been studied. 

In 1987, D. D. Scott \cite{Scott} 
introduced the notion of double competition graphs 
as a variant of 
the notion of competition graphs. 
The \emph{double competition graph} 
(or the \emph{competition-common enemy graph} or the \emph{CCE graph})
of a digraph $D$ 
is the graph which has the same vertex set as $D$ 
and has an edge between two distinct vertices $x$ and $y$ 
if and only if both 
$N^+_D(x) \cap N^+_D(y) \neq \emptyset$ and 
$N^-_D(x) \cap N^-_D(y) \neq \emptyset$ hold. 
See \cite{KKR07, LW09, Sano10, WL10} for recent results on double competition graphs. 

A \emph{multigraph} $M$ is a pair $(V(M),E(M))$
of a set $V(M)$ of \emph{vertices} 
and a multiset $E(M)$ of unordered pairs of vertices, called \emph{edges}. 
Note that, in our definition, multigraphs have no loops. 
We may consider a multigraph $M$ as 
the pair $(V(M),m_M)$ of the vertex set $V(M)$ and 
the nonnegative integer-valued function 
$m_M: {V \choose 2} \to \mathbb{Z}_{\geq 0}$ 
on the set ${V \choose 2}$ of all unordered pairs of $V$ 
where $m_M(\{x,y\})$ is defined to be the number of multiple edges between 
the vertices $x$ and $y$ in $M$. 
The notion of competition multigraphs was introduced by 
C. A. Anderson, K. F. Jones, J. R. Lundgren, and T. A. McKee 
\cite{AJLM90} in 1990 as a variant of 
the notion of competition graphs. 
The \emph{competition multigraph} of a digraph $D$ 
is the multigraph which has the same vertex set as $D$ 
and has $m_{xy}$ multiple edges between two distinct vertices $x$ and $y$, 
where $m_{xy}$ is the nonnegative integer defined by 
$m_{xy} = |N^+_D(x) \cap N^+_D(y)|$. 
See \cite{Sano09, ZC10} for recent results on competition multigraphs. 

In this article, we introduce the notion of 
the double competition multigraph of a digraph, and 
we give characterizations of the double competition 
multigraphs of arbitrary digraphs, 
loopless digraphs, reflexive digraphs, 
and acyclic digraphs 
in terms of edge clique partitions of the multigraphs. 

\section{Main Results}

We define the double competition multigraph of a digraph
as follows. 

\begin{defi}
Let $D$ be a digraph. 
The \emph{double competition multigraph} of $D$ 
is the multigraph which has the same vertex set as $D$ 
and has $m_{xy}$ multiple edges between two distinct vertices $x$ and $y$, 
where $m_{xy}$ is the nonnegative integer defined by 
\[
m_{xy} = |N^+_D(x) \cap N^+_D(y)| 
\cdot |N^-_D(x) \cap N^-_D(y)|, 
\]
i.e., the multigraph $M$ defined by 
$V(M)=V(D)$ and $m_M(\{x,y\})=m_{xy}$. 
\qed
\end{defi}

Recall that a \emph{clique} of a multigraph $M$ is 
a set of vertices of $M$ which are pairwise adjacent. 
We consider the empty set $\emptyset$ as a clique of any multigraph 
for convenience. 
A multiset is also called a \emph{family}. 
An \emph{edge clique partition} of a multigraph $M$ 
is a family $\mathcal{F}$ of cliques of $M$ 
such that any two distinct vertices $x$ and $y$ 
are contained in exactly $m_M(\{x,y\})$ cliques 
in the family $\mathcal{F}$. 
For a positive integer $n$, 
let $[n]$ denote the set $\{1,2, \ldots, n\}$. 


\begin{theorem}\label{thm:1}
Let $M$ be a multigraph with $n$ vertices. 
Then, 
$M$ is the double competition multigraph of an arbitrary digraph 
if and only if 
there exist an ordering $(v_1, \ldots, v_n)$ 
of the vertices of $M$ 
and a double indexed edge clique partition 
$\{ S_{ij} \mid i,j \in [n] \}$ of $M$ 
such that the following condition holds: 
\begin{itemize}
\item[{\rm (I)}] 
for any $i,j \in [n]$, 
if $|A_i \cap B_j| \geq 2$, then $A_i \cap B_j = S_{ij}$, 
\end{itemize}
where $A_i$ and $B_j$ are the sets defined by  
\begin{equation}\label{eq:Ai} 
A_i = S_{i *} \cup T^+_i,  
\quad
S_{i *} := \bigcup_{p \in [n]} S_{ip},  
\quad
T^+_i := \{ v_b \mid a,b \in [n], v_i \in S_{ab} \}, 
\end{equation}
\begin{equation}\label{eq:Bj} 
B_j = S_{* j} \cup T^-_j,  
\quad 
S_{* j} := \bigcup_{q \in [n]} S_{qj},  
\quad 
T^-_j := \{ v_a \mid a,b \in [n], v_j \in S_{ab} \}. 
\end{equation}
\end{theorem}

\begin{proof} 
First, we show the only-if part. 
Let $M$ be the double competition multigraph of an arbitrary digraph $D$. 
Let $(v_1, \ldots, v_n)$ be an ordering of the vertices of $D$. 
For $i,j \in [n]$, we define 
\begin{equation}\label{eq:Sij}
S_{ij} := \{v_k \in V(D) \mid (v_i, v_k), (v_k, v_j) \in A(D) \}. 
\end{equation}
Then $S_{ij}$ is 
a clique of $M$. 
Let $\mathcal{F}$ be the family of $S_{ij}$'s whose size is at least two, 
i.e., 
\begin{equation}\label{eq:F}
\mathcal{F} := \{S_{ij} \mid i,j \in [n], |S_{ij}| \geq 2 \}. 
\end{equation}
By the definition of a double competition multigraph, 
$\mathcal{F}$ is an edge clique partition of $M$. 

We show that the condition (I) holds.
Fix $i$ and $j$ in $[n]$ and let 
$A_i$ and $B_j$ be sets as defined in (\ref{eq:Ai}) and (\ref{eq:Bj}). 
Since $S_{ij} \subseteq A_i$ and 
$S_{ij} \subseteq B_j$, it holds that 
$S_{ij} \subseteq A_i \cap B_j$. 
Now we assume that $|A_i \cap B_j| \geq 2$ 
and take any vertex $v_k \in A_i \cap B_j$. 
There are four cases for $v_k$ arising from the definitions 
of $A_i$ and $B_j$ as follows: 
(i) 
$v_k \in S_{i *} \cap S_{* j}$; 
(ii) 
$v_k \in S_{i *} \cap T^-_j$; 
(iii) 
$v_k \in T^+_i \cap S_{* j}$; 
(iv) 
$v_k \in T^+_i \cap T^-_j$. 
To show $A_i \cap B_j \subseteq S_{ij}$, 
we will check that $v_k \in S_{ij}$ for each case. 
Consider the case (i). 
Since $v_k \in S_{i *}$, 
there exists $p \in [n]$ such that 
$v_k \in S_{ip}$. 
Since $v_k \in S_{* j}$, 
there exists $q \in [n]$ such that 
$v_k \in S_{qj}$. 
By (\ref{eq:Sij}), 
$v_k \in S_{ip}$ implies 
$(v_i, v_k), (v_k, v_p) \in A(D)$, 
and $v_k \in S_{qj}$ implies 
$(v_q, v_k), (v_k, v_j) \in A(D)$. 
Therefore we have  
$(v_i, v_k), (v_k, v_j) \in A(D)$, 
which implies $v_k \in S_{ij}$. 
Consider the case (ii).  
Since $v_k \in S_{i *}$, 
there exists $p \in [n]$ such that 
$v_k \in S_{ip}$. 
Since $v_k \in T^-_j$, 
there exists $b \in [n]$ 
such that $v_j \in S_{kb}$. 
By (\ref{eq:Sij}), 
$v_k \in S_{ip}$ implies 
$(v_i, v_k), (v_k, v_p) \in A(D)$, 
and $v_j \in S_{kb}$ implies 
$(v_k, v_j), (v_j, v_b) \in A(D)$. 
Therefore we have 
$(v_i, v_k), (v_k, v_j) \in A(D)$, 
which implies $v_k \in S_{ij}$. 
Consider the case (iii).  
Since $v_k \in T^+_i$, 
there exists $a \in [n]$ 
such that $v_i \in S_{ak}$. 
Since $v_k \in S_{* j}$, 
there exists $q \in [n]$ such that 
$v_k \in S_{qj}$. 
By (\ref{eq:Sij}), 
$v_i \in S_{ak}$ implies 
$(v_a, v_i), (v_i, v_k) \in A(D)$, 
and 
$v_k \in S_{qj}$ implies 
$(v_q, v_k), (v_k, v_j) \in A(D)$. 
Therefore we have 
$(v_i, v_k), (v_k, v_j) \in A(D)$, 
which implies $v_k \in S_{ij}$. 
Consider the case (iv).  
Since $v_k \in T^+_i$, 
there exists $a \in [n]$ 
such that $v_i \in S_{ak}$. 
Since $v_k \in T^-_j$, 
there exists $b \in [n]$ 
such that $v_j \in S_{kb}$. 
By (\ref{eq:Sij}), 
$v_i \in S_{ak}$ implies 
$(v_a, v_i), (v_i, v_k) \in A(D)$, 
and 
$v_j \in S_{kb}$ implies 
$(v_k, v_j), (v_j, v_b) \in A(D)$. 
Therefore we have 
$(v_i, v_k), (v_k, v_j) \in A(D)$, 
which implies $v_k \in S_{ij}$. 
Thus we obtain $A_i \cap B_j \subseteq S_{ij}$, 
and so $A_i \cap B_j = S_{ij}$. 
Hence the condition (I) holds. 

Next, we show the if part. 
Let $M$ be a multigraph with $n$ vertices, and  
suppose that 
there exist an ordering $(v_1, \ldots, v_n)$ 
of the vertices of $M$ 
and a double indexed edge clique partition 
$\mathcal{F} = \{ S_{ij} \mid i,j \in [n] \}$ of $M$ 
such that the condition (I) holds. 

We define a digraph $D$ by 
$V(D) := V(M)$ and 
\begin{equation}\label{eq:AD}
A(D) := 
\bigcup_{i,j \in [n]} 
\left(
\bigcup_{v_k \in S_{ij}} 
\{(v_i, v_k), (v_k, v_j) \} 
\right). 
\end{equation}
Let $M'$ denote the double competition multigraph of $D$. 
We show that $M=M'$. 
Since $V(M)=V(M')$, it is enough to show $m_{M}=m_{M'}$. 
Take any two distinct vertices $v_k$ and $v_l$ 
and let $t:=m_M(\{v_{k}, v_{l}\})$. 
Since $\mathcal{F}$ is an edge clique partition of $M$, 
the vertices $v_k$ and $v_l$ are contained in 
exactly $t$ cliques $S_{ij} \in \mathcal{F}$. 
So, for some nonnegative integers $r$ and $s$ with $rs=t$, 
there are $r$ common in-neighbors $v_{i_1}, \ldots, v_{i_r}$ 
and $s$ common out-neighbors $v_{j_1}, \ldots, v_{j_s}$ 
of the vertices $v_k$ and $v_l$ in $D$. 
Therefore it follows that $m_{M'}(\{v_k, v_l\}) 
= |N^-_D(v_k) \cap N^-_D(v_l)| \cdot |N^+_D(v_k) \cap N^+_D(v_l)| 
\geq rs =t$. 
Thus $m_{M}(\{v_k, v_l\}) \leq m_{M'}(\{v_k, v_l\})$. 
Again, take any two distinct vertices $v_k$ and $v_l$ 
and let $t':=m_{M'}(\{v_{k}, v_{l}\})$. 
Then, for some nonnegative integers $r'$ and $s'$ with $r's'=t'$, 
there are $r'$ common in-neighbors $v_{i_1}, \ldots, v_{i_{r'}}$ 
and $s'$ common out-neighbors $v_{j_1}, \ldots, v_{j_{s'}}$ 
of the vertices $v_k$ and $v_l$ in $D$. 
For each $i \in \{i_1, \ldots, i_{r'}\}$, 
since $(v_i, v_k), (v_i, v_l) \in A(D)$, 
it follows that $\{v_k, v_l\} \subseteq A_i$. 
Similarly, for each $j \in \{j_1, \ldots, j_{s'}\}$, 
since $(v_k, v_j), (v_l, v_j) \in A(D)$, 
it follows that $\{v_k, v_l\} \subseteq B_j$. 
Therefore, $\{v_k, v_l\} \subseteq A_i \cap B_j$ 
for any $i \in \{i_1, \ldots, i_{r'}\}$ and 
any $j \in \{j_1, \ldots, j_{s'}\}$. 
By the condition (I), we have $A_i \cap B_j = S_{ij}$. 
Therefore $\{v_k, v_l\} \subseteq S_{ij}$ 
for any $i \in \{i_1, \ldots, i_{r'}\}$ and 
any $j \in \{j_1, \ldots, j_{s'}\}$ 
and this implies that $m_M(\{v_k, v_l\}) = 
|\{S_{i,j} \in \mathcal{F} \mid \{v_k, v_l\} \subseteq S_{i,j} \}| 
\geq r's'=t'$. 
Thus $m_{M}(\{v_k, v_l\}) \geq m_{M'}(\{v_k, v_l\})$. 
Hence it holds that $m_{M}(\{v_k, v_l\}) = m_{M'}(\{v_k, v_l\})$ 
for any two distinct vertices $v_k$ and $v_l$, that is, 
$m_M = m_{M'}$, i.e., $M=M'$. 
So $M$ is the double competition multigraph of $D$. 
\end{proof}


A digraph $D$ is said to be \emph{loopless} 
if $D$ has no loops, i.e., 
$(v,v) \not\in A(D)$ holds for any $v \in V(D)$. 

\begin{theorem}\label{thm:2}
Let $M$ be a multigraph with $n$ vertices. 
Then, 
$M$ is the double competition multigraph of a loopless digraph 
if and only if 
there exist an ordering $(v_1, \ldots, v_n)$ 
of the vertices of $M$ 
and a double indexed edge clique partition 
$\{ S_{ij} \mid i,j \in [n] \}$ of $M$ 
such that the following conditions hold: 
\begin{itemize}
\item[{\rm (I)}] 
for any $i,j \in [n]$, if $|A_i \cap B_j| \geq 2$, then $A_i \cap B_j = S_{ij}$; 
\item[{\rm (II)}] 
for any $i,j \in [n]$, $v_i \not\in S_{ij}$ and $v_j \not\in S_{ij}$, 
\end{itemize}
where 
$A_i$ and $B_j$ are the sets defined as 
{\rm (\ref{eq:Ai})} and {\rm (\ref{eq:Bj})}. 
\end{theorem} 

\begin{proof}
First, we show the only-if part. 
Let $M$ be the double competition multigraph of a loopless digraph $D$. 
Let $(v_1, \ldots, v_n)$ be an ordering of the vertices of $D$. 
Let $S_{ij}$ $(i,j \in [n])$ be the sets defined as (\ref{eq:Sij}), 
and let $\mathcal{F}$ be the family defined as (\ref{eq:F}). 
Then $S_{ij}$ is 
a clique of $M$, 
and $\mathcal{F}$ is an edge clique partition of $M$. 
Moreover, we can show, as in the proof of Theorem \ref{thm:1}, 
that the condition (I) holds. 
Now we show that the condition (II) holds. 
Take any vertex $v_{k} \in S_{ij}$. 
Then $(v_i, v_k), (v_k, v_j) \in A(D)$. 
Since $D$ is loopless, we have $v_{i} \neq v_{k}$ and $v_{i} \neq v_{k}$. 
Therefore it follows that $v_{i} \not\in S_{ij}$ and $v_{j} \not\in S_{ij}$. 
Thus the condition (II) holds. 

Next, we show the if part. 
Let $M$ be a multigraph with $n$ vertices, and  
suppose that 
there exists an ordering $(v_1, \ldots, v_n)$ 
of the vertices of $M$ 
and a double indexed edge clique partition 
$\{ S_{ij} \mid i,j \in [n] \}$ of $M$ 
such that the conditions (I) and (II) hold. 
We define a digraph $D$ by 
$V(D) := V(M)$ and $A(D)$ given in (\ref{eq:AD}). 
By the condition (II), 
it follows from the definition of $D$ that 
$(v_i, v_i) \not\in A(D)$ for any $i \in [n]$. 
Therefore $D$ is a loopless digraph. 
Moreover we can show, as in the proof of Theorem \ref{thm:1}, 
that $M$ is the double competition multigraph of $D$. 
\end{proof}


A digraph $D$ is said to be \emph{reflexive} 
if all the vertices of $D$ have loops, i.e., 
$(v,v) \in A(D)$ holds for any $v \in V(D)$. 

\begin{theorem}\label{thm:3}
Let $M$ be a multigraph with $n$ vertices. 
Then, 
$M$ is the double competition multigraph of a reflexive digraph 
if and only if 
there exist an ordering $(v_1, \ldots, v_n)$ 
of the vertices of $M$ 
and a double indexed edge clique partition 
$\{ S_{ij} \mid i,j \in [n] \}$ of $M$ 
such that the following conditions hold: 
\begin{itemize}
\item[{\rm (I)}] 
for any $i,j \in [n]$, if $|A_i \cap B_j| \geq 2$, then $A_i \cap B_j = S_{ij}$; 
\item[{\rm (III)}] 
for any $i \in [n]$, $v_i \in S_{i*} \cup S_{*i}$, 
\end{itemize}
where 
$A_i$, $B_j$, $S_{i*}$, and $S_{*i}$ are the sets defined as 
{\rm (\ref{eq:Ai})} and {\rm (\ref{eq:Bj})}. 
\end{theorem} 

\begin{proof}
First, we show the only-if part. 
Let $M$ be the double competition multigraph of a reflexive digraph $D$. 
Let $(v_1, \ldots, v_n)$ be an ordering of the vertices of $D$. 
Let $S_{ij}$ $(i,j \in [n])$ be the sets defined as (\ref{eq:Sij}), 
and let $\mathcal{F}$ be the family defined as (\ref{eq:F}). 
Then $S_{ij}$ is 
a clique of $M$, 
and $\mathcal{F}$ is an edge clique partition of $M$. 
Moreover, we can show, as in the proof of Theorem \ref{thm:1}, 
that the condition (I) holds. 
Now we show that the condition (III) holds. 
Since $D$ is reflexive, we have $(v_i, v_i) \in A(D)$ for any $i \in [n]$. 
Then it follows that 
there exists $p \in [n]$ such that $v_i \in S_{ip}$ or $v_i \in S_{pi}$. 
Therefore $v_{i} \in S_{i*} \cup S_{*i}$. 
Thus the condition (III) holds. 

Next, we show the if part. 
Let $M$ be a multigraph with $n$ vertices, and  
suppose that 
there exist an ordering $(v_1, \ldots, v_n)$ 
of the vertices of $M$ 
and a double indexed edge clique partition 
$\mathcal{F} = \{ S_{ij} \mid i,j \in [n] \}$ of $M$ 
such that the conditions (I) and (III) hold. 
We define a digraph $D$ by 
$V(D) := V(M)$ and $A(D)$ given in (\ref{eq:AD}). 
Fix any $i \in [n]$. 
By the condition (III), 
there exists $p \in [n]$ such that $v_i \in S_{ip}$ or $v_i \in S_{pi}$. 
Then it follows from the definition of $D$ that 
$(v_i, v_i) \in A(D)$. 
Therefore $D$ is a reflexive digraph. 
Moreover we can show, as in the proof of Theorem \ref{thm:1}, 
that $M$ is the double competition multigraph of $D$. 
\end{proof}


A digraph $D$ is said to be \emph{acyclic} 
if $D$ has no directed cycles. 
An ordering $(v_1, \ldots, v_{n})$ 
of the vertices of a digraph $D$, 
where $n$ is the number of vertices of $D$, 
is called an \emph{acyclic ordering} of $D$ 
if $(v_i,v_j) \in A(D)$ implies $i<j$. 
It is well known that 
a digraph $D$ is acyclic if and only if 
$D$ has an acyclic ordering. 

\begin{theorem}\label{thm:4}
Let $M$ be a multigraph with $n$ vertices. 
Then, 
$M$ is the double competition multigraph of an acyclic digraph 
if and only if 
there exist an ordering $(v_1, \ldots, v_n)$ 
of the vertices of $M$ 
and a double indexed edge clique partition 
$\{ S_{ij} \mid i,j \in [n] \}$ of $M$ 
such that the following conditions hold: 
\begin{itemize}
\item[{\rm (I)}] 
for any $i,j \in [n]$, if $|A_i \cap B_j| \geq 2$, then $A_i \cap B_j = S_{ij}$; 
\item[{\rm (IV)}] 
for any $i,j,k \in [n]$, $v_k \in S_{ij}$ implies $i < k < j$, 
\end{itemize}
where 
$A_i$ and $B_j$ are the sets defined as 
{\rm (\ref{eq:Ai})} and {\rm (\ref{eq:Bj})}. 
\end{theorem} 

\begin{proof}
First, we show the only-if part. 
Let $M$ be the double competition multigraph of an acyclic digraph $D$. 
Let $(v_1, \ldots, v_n)$ be an acyclic ordering of the vertices of $D$. 
Let $S_{ij}$ $(i,j \in [n])$ be the sets defined as (\ref{eq:Sij}), 
and let $\mathcal{F}$ be the family defined as (\ref{eq:F}). 
Then $S_{ij}$ is 
a clique of $M$, 
and $\mathcal{F}$ is an edge clique partition of $M$. 
Moreover, we can show, as in the proof of Theorem \ref{thm:1}, 
that the condition (I) holds. 
Now we show that the condition (IV) holds. 
Suppose that $v_{k} \in S_{ij}$. 
Then $(v_i, v_k), (v_k, v_j) \in A(D)$. 
Since $(v_1, \ldots, v_n)$ is an acyclic ordering of $D$, 
$(v_i, v_k) \in A(D)$ implies $i < k$ 
and 
$(v_k, v_j) \in A(D)$ implies $k < j$ 
Therefore $i < k < j$. 
Thus the condition (IV) holds. 

Next, we show the if part. 
Let $M$ be a multigraph with $n$ vertices, and  
suppose that 
there exist an ordering $(v_1, \ldots, v_n)$ 
of the vertices of $M$ 
and a double indexed edge clique partition 
$\{ S_{ij} \mid i,j \in [n] \}$ of $M$ 
such that the conditions (I) and (IV) hold. 
We define a digraph $D$ by 
$V(D) := V(M)$ and $A(D)$ given in (\ref{eq:AD}). 
By the condition (IV), 
it follows from the definition of $D$ that 
$(v_1, \ldots, v_n)$ is an acyclic ordering of $D$. 
Therefore $D$ is an acyclic digraph. 
Moreover we can show, as in the proof of Theorem \ref{thm:1}, 
that $M$ is the double competition multigraph of $D$. 
\end{proof}


\begin{rem}
The condition (I) in 
Theorems \ref{thm:1}, \ref{thm:2}, \ref{thm:3}, and \ref{thm:4} 
may be replaced by the following condition:  
\begin{itemize}
\item[{\rm (I)$'$}]
for any $i,j \in [n]$, $A_i \cap B_j = S_{ij}$. 
\end{itemize}
\end{rem}

\begin{proof}
If the condition (I)$'$ is satisfied, then so is 
the condition (I). 
Suppose that the condition (I) is satisfied. 
If $|A_i \cap B_j| \geq 2$, then 
it follows from the condition (I) that $A_i \cap B_j = S_{ij}$.  
If $|A_i \cap B_j| = 0$, then $A_i \cap B_j= \emptyset$. 
Since $S_{ij} \subseteq A_i \cap B_j$, we have $S_{ij} = \emptyset$. 
Therefore, $A_i \cap B_j = S_{ij}$. 
If $|A_i \cap B_j| = 1$, then $A_i \cap B_j= \{v_k\}$ for some $k \in [n]$. 
Since $S_{ij} \subseteq A_i \cap B_j$, 
we have $S_{ij} = \emptyset$ or $S_{ij} = \{v_k\}$. 
If $S_{ij} = \{v_k\}$, then $A_i \cap B_j = S_{ij}$. 
If $S_{ij} = \emptyset$, then we replace 
$S_{ij} = \emptyset$ by $S_{ij} = \{v_k\}$. 
Then $\mathcal{F}$ is still an edge clique partition of $M$, 
and $A_i \cap B_j = S_{ij}$. Thus the condition (I)$'$ holds. 
Hence the remark holds. 
\end{proof}


At the end of this paper, 
we mention two corollaries related to 
the edge clique partition number of a multigraph. 
Recall that 
the \emph{edge clique partition number} of a multigraph $M$ 
is the minimum size of an edge clique partition of $M$ 
and is denoted by $\theta^{*}(M)$.  
As a corollary of Theorem \ref{thm:1}, 
we obtain a necessary condition for multigraphs 
being the double competition multigraph of a digraph. 

\begin{cor}
If a multigraph $M$ with $n$ vertices 
is the double competition multigraph of a digraph, 
then $\theta^{*}(M) \leq n^2$. 
\end{cor}

For the double competition multigraphs of acyclic digraphs, 
we can improve the above upper bound for 
the edge clique partition numbers of multigraphs. 

\begin{cor}
If a multigraph $M$ with $n$ vertices 
is the double competition multigraph of an acyclic digraph, then 
$\theta^{*}(M) \leq \frac{1}{2}(n-2)(n-3)$.  
\end{cor}

\begin{proof}
Suppoe that a multigraph $M$ with $n$ vertices 
is the double competition multigraph of an acyclic digraph. 
Then, by Theorem \ref{thm:4}, 
there exist an ordering $(v_1, \ldots, v_n)$ 
of the vertices of $M$ 
and a double indexed edge clique partition 
$\{ S_{ij} \mid i,j \in [n] \}$ of $M$ 
satisfying the conditions (I) and (IV). 
It follows from the condition (IV) that, 
if $j \leq i+1$, then $S_{ij}=\emptyset$. 
If $j=i+2$, then $S_{ij}=\emptyset$ or $S_{ij}=\{v_{i+1}\}$, 
which does not cover an edge of $M$. 
Therefore, the family 
$\{ S_{ij} \mid i,j \in [n], i+3 \leq j \}$ 
is an edge clique partition of $M$. 
Thus the corollary holds.  
\end{proof}

\begin{rem}
In \cite{PS14}, the authors defined 
the \emph{double multicompetition number} $dk^*(M)$ 
of a multigraph $M$ 
to be the minimum nonnegative integer $k$ 
such that $M$ together with $k$ new isolated vertices is 
the double competition multigraph of some acyclic digraph. 
In this context, Theorem \ref{thm:4} 
gives a characterization of multigraphs 
whose double multicompetition number is equal to $0$. 
\end{rem}


\end{document}